\documentclass [a4paper, 10pt] {article}
\usepackage[cp1251]{inputenc}
\usepackage{amsmath,amssymb,longtable}
\usepackage[english]{babel}
\usepackage[final]{graphicx}
\usepackage{hyperref}
\usepackage{latexsym}
\usepackage{amsthm}
\usepackage{url}

\usepackage[all]{xy}

\usepackage{epstopdf}
\DeclareGraphicsExtensions{.pdf,.eps,.png,.jpg,.mps}

\newtheorem{theorem}{Theorem}
\newtheorem{lemma}[theorem]{Lemma}

\newtheorem{cor}[theorem]{Corollary}

\newcounter{listingcounter}

\title{\bf On the number of\\semi-magic squares of order 6}

\author{Artem Ripatti\\
\small Ufa State Aviational Technical University\\[-0.8ex]
\small M.Akmullah Bashkir State Pedagogical University\\[-0.8ex]
\small Ufa, Russia\\
\small\tt ripatti@inbox.ru
}

\date{\today}

\begin{document}

\maketitle

\textbf{Abstract}: We present an exact method for counting semi-magic squares of order 6. Some theoretical investigations about the number of them and a probabilistic method are presented. Our calculations show that there are exactly $94\,590\,660\,245\,399\,996\,601\,600$ such squares up to reflections and rotations.

\medskip


\section*{Introduction}

\textit{A semi-magic square} is a matrix of size $n \times n$ filled by numbers $1, \cdots , n^2$ in such a way that each number $1, \cdots , n^2$ is used exactly once and the sum over elements of each row and each column is equal to \textit{the magic constant} $Z=n(n^2+1)/2$. A semi-magic square is \textit{magic} if sums over elements of main diagonals are equal to the magic constant too. A magic square is \textit{panmagic} (\textit{pandiagonal}, \textit{diabolical}, \textit{Nasik}) if sums over elements of all $2(n-1)$ broken diagonals are also equal to $Z$.  Here are examples of semi-magic, magic and panmagic squares of order 5:
$$
\left[\begin{matrix}
4 & 12 & 25 & 8 & 16 \\
23 & 6 & 19 & 2 & 15 \\
11 & 24 & 7 & 20 & 3 \\
17 & 5 & 13 & 21 & 9 \\
10 & 18 & 1 & 14 & 22 \\
\end{matrix}\right]
\quad
\left[\begin{matrix}
13 & 11 & 15 & 16 & 10 \\
17 & 5 & 19 & 18 & 6 \\
14 & 22 & 1 & 8 & 20 \\
9 & 24 & 7 & 21 & 4 \\
12 & 3 & 23 & 2 & 25 \\
\end{matrix}\right]
\quad
\left[\begin{matrix}
1 &  15 & 24 & 8 & 17 \\
23 & 7 & 16 & 5 & 14 \\
20 & 4 & 13 & 22 & 6 \\
12 & 21 & 10 & 19 & 3 \\
9 & 18 & 2 & 11 & 25 \\
\end{matrix}\right]
$$

Table~\ref{table_enum} shows the currently known results on enumerating squares of these types. We remark that squares that can be transformed into one another by reflections and rotations are considered to be the same, and the table represents only the number of essential different squares relative to these operations. More results can be found on the website of Walter Trump~\cite{trump2018}.
\begin{table}[htp]
\caption{Known the numbers of squares of different types}
\label{table_enum}\centering\small%
\begin{tabular}{ccccc}
$n$ & semi-magic & magic & panmagic \\
& A & B & C \\
\hline
$3$ & $9$ & $1$ & 0 \\
$4$ & $68\,688$ & $880$ & $48$ \\
$5$ & $579\,043\,051\,200$ & $275\,305\,224$ & $3600$ \\
$6$ & $9.459156(51) \cdot 10^{22}$ &  $1.775399(42) \cdot 10^{19}$ & $0$ \\
$7$ & $4.2848(17) \cdot 10^{38}$ & $3.79809(50) \cdot 10^{34}$ & $1.21(12) \cdot 10^{17}$ \\
\hline
\end{tabular}

Here $1.2345(67) \cdot 10^{8}$ means that the exact value fits into the interval

$(1.2345 \pm 0.0067) \times 10^8$ with probability $99.7\%$ ($3 \sigma$ error).
\end{table}

The unique magic square of order $3$ (B3) was found by Luo Shu (as it is known). The oldest known image of this square on a tortoise shell is dated 2200~BCE. B4 was found by Bernard Fr\'enicle de Bessy in 1693~\cite{de1693quarrez}. First analytic proof of B4 was made by Kathleen Ollerenshaw and Hermann Bondi in 1982~\cite{ollerenshaw1982magic}. Richard Schroeppel calculated B5 in 1973 using a computer program. His result was published in 1976 by Martin Gardner~\cite{gardner1976}.

C5 is equal to the number of \textit{regular panmagic squares} which can be generated using Graeco-Latin squares, as Leonhard Euler pointed in 1779 (published in 1782~\cite{euler1782}). There are no panmagic squares of order 6 (C6), which has been proved by Andrew Hollingworth Frost in 1878~\cite{frost1878}, and more elegantly by Charles~Planck in 1919~\cite{planck1919}.

A4 was found by Mutsumi Suzuki in 1997~\cite{suzuki1997}. Walter Trump calculated A5 in 2000~\cite{trump2018}.

In 1998 Klaus Pinn and Christian Wieczerkowski~\cite{pinn1998} estimated B6 as $(0.17745 \pm 0.00016) \times 10^{20}$ using a method from statistical physics called ``parallel tempering Monte Carlo''. They also estimated B7 as $(0.3760 \pm 0.0052) \times 10^{35}$ using the same method. Walter Trump improved on these results using ``Monte Carlo backtracking'' method, see Table~\ref{table_enum}. He also estimated A6, A7 and C7~\cite{trump2018}. We remark that in~\cite{trump2018} the estimate of A6 is already replaced by the exact result.

In this paper we focus on A6: the number of semi-magic squares of order 6. We present an exact method for calculating this number which is a variation of the meet-in-the-middle approach. The method is described in Section~\ref{sec_exact}. We use this method to calculate A6, and in order to partially verify it we do some theoretical investigations in Section~\ref{sec_theo} and describe a probabilistic method in Section~\ref{sec_prob}. The computing results are presented in Section~\ref{sec_results}. Finally, we draw a conclusion in Section~\ref{sec_conc}.

\section{Exact method} \label{sec_exact}

Let $S$ be a semi-magic square. Then $S(i,j)$ is the number on the intersection of row $i$ and column $j$, $r_i(S)$ is the set of numbers of $S$ in row $i$, and $c_j(S)$ is the set of numbers of $S$ in column $j$. $U(S) = r_1(S) \cup r_2(S) \cup r_3(S)$ and $L(S) = r_4(S) \cup r_5(S) \cup r_6(S)$ are the upper and lower halves of $S$.

Let $A$ be $\{ 1, 2, \cdots , 36 \}$ and $Q$ be the set of all semi-magic squares of order 6 up to reflections and rotations. Our goal is to find $|Q|$.

\subsection{Canonical semi-magic squares}

We transform each semi-magic square into \textit{a canonical} form. First, we sort all rows by the their minimal element. After that, we sort all columns by the sum of their upper triplet. In case of a tie we sort them by the upper element. The following picture shows an example of such a transformation:
$$
\left[\begin{matrix}
\bf{1} &  4 & 16 & 32 & 27 & 31 \\
34 & 35 &  7 & 19 & 13 & \bf{3} \\
15 & \bf{12} & 20 & 28 & 14 & 22 \\
29 & 18 & 24 & \bf{2} & 30 &  8 \\
26 & \bf{9} & 23 & 25 & 17 & 11 \\
 6 & 33 & 21 & \bf{5} & 10 & 36 \\
\end{matrix}\right]
\to
\left[\begin{matrix}
\bf{1} &  4 & 16 & 32 & 27 & 31 \\
29 & 18 & 24 & \bf{2} & 30 &  8 \\
34 & 35 &  7 & 19 & 13 & \bf{3} \\
 6 & 33 & 21 & \bf{5} & 10 & 36 \\
26 & \bf{9} & 23 & 25 & 17 & 11 \\
15 & \bf{12} & 20 & 28 & 14 & 22 \\
\end{matrix}\right]
\to
$$
$$\text{sums over upper triplets are }(64, 57, 47, 53, 70, 42)$$
$$
\to
\left[\begin{matrix}
31 & 16 & 32 & 4 & \bf{1} & 27 \\
8 & 24 & \bf{2} & 18 & 29 & 30 \\
\bf{3} & 7 & 19 & 35 & 34 & 13 \\
36 & 21 & \bf{5} & 33 & 6 & 10 \\
11 & 23 & 25 & \bf{9} & 26 & 17 \\
22 & 20 & 28 & \bf{12} & 15 & 14 \\
\end{matrix}\right]
$$

Each canonical semi-magic square represents exactly $(6!)^2$ semi-magic squares. We can calculate the number of canonical semi-magic squares and then multiply it by $(6!)^2/8$ to get the total number of semi-magic squares $|Q|$ up to reflections and rotations.

\subsection{Classes of canonical semi-magic squares}

Let $C$ be the set of all canonical semi-magic squares of order 6. We split the set $C$ into subsets $C_m$, $m \in M$, where $M = \{ x \; | \; x \subset A, |x|=18 \}$ and $m = U(c)$ for all $c \in C_m$.

$|M|=C_{36}^{18} \approx 9 \times 10^9$, but $|C_m| = 0$ for most $m \in M$. For example, $\sum_{i \in m}i$ must be $333$. There are exactly $113\,093\,022$ such sets.

We consider a set $M' \subset M$ such that each its element $m \in M'$ has the following property: there exist two partitions $m = r_1 \cup r_2 \cup r_3$ and $A \backslash m = r_4 \cup r_5 \cup r_6$ ($r_i \cap r_j = \varnothing$ for all $i \neq j$) such that:

\begin{enumerate}
\item $|r_i|=6$ for all $1 \le i \le 6$,
\item $\sum_{a \in r_i}a=111$ for all $1 \le i \le 6$,
\item $\min{r_1} < \min{r_2} < \min{r_3} < \min{r_4} < \min{r_5} < \min{r_6}$.
\end{enumerate}

$|C_m|=0$ for all $m \in M \backslash M'$.

\subsection{Generating $M'$}

First, we generate a set $S = \{ s \; | \; s \subset A, |s|=6, \sum_{i \in s}i=111 \}$ using a simple backtracking algorithm. Each element of $S$ is called {\it a series}, and $|S|=32\,134$.

Then, we generate a set $M^* \subset M$ such that $\sum_{i \in m}i = 333$ for all $m \in M^*$, also using a simple backtracking algorithm. Now we need to check all $113\,093\,022$ such sets according to rules described in the previous subsection.

We do the check in a very straightforward way: we try splitting the current $m \in M^*$ into series $s_1$, $s_2$ and $s_3$ ($s_1,s_2,s_3 \in S$, $s_1 \cup s_2 \cup s_3 = m$), such that $\min s_1 < \min s_2 < \min s_3 < \min A \backslash m$. If we have at least one such split, we additionally check that $A \backslash m$ can be split into series $s_4$, $s_5$ and $s_6$ ($s_4,s_5,s_6 \in S$, $s_4 \cup s_5 \cup s_6 = A \backslash m$), such that $\min s_4 < \min s_5 < \min s_6$.

There are two tricks to make the check faster:

{\bf Trick 1.} The straightforward way to check that $m \in M^*$ can be split into three series $s_1$, $s_2$ and $s_3$ is running two nested loops that iterate over series $s_1$ and $s_2$ from $S$ while checking that $s_1 \subset m$, $s_2 \subset m$, $s_1 \cap s_2 = \varnothing$ (obviously, $s_3 = m \backslash s_1 \backslash s_2$ is a series too). The trick is to replace $S$ by $S' = \{ s \; | \; s \in S, s \subset m \}$ here.

{\bf Trick 2.} All operations on sets can be implemented with binary masks. Namely, the $\min$ operation can be implemented in the following way:
\begin{verbatim}
min_bit = (mask & (mask - 1)) ^ mask;
\end{verbatim}
Actually, it computes not the minimal element, but a mask with a single bit set that corresponds to the minimal element. But $a < b$ iff $2^a < 2^b$, so we can use it for comparison.

Mixing all things together, we can generate the set $M'$ in about $4$ hours.

The computation shows that $|M'| = 9\,366\,138$.

Now, we need to determine $|C_m|$ for all $m \in M'$.

\subsection{Counting squares of a single class}

Consider some $m \in M'$. First, we generate all partitions of $m$ into three series $s_1$, $s_2$ and $s_3$, such that $\min s_1 < \min s_2 < \min s_3 < \min A \backslash m$. Let $U$ be the set of all such partitions. Similarly, we generate a set $L$ of all partitions of $A \backslash m$ into three series $s_4$, $s_5$ and $s_6$, such that $\min s_4 < \min s_5 < \min s_6$. To generate them fast we use the tricks described in the previous subsection.

For each $(s_1, s_2, s_3) \in U$ we can generate $720 \times 720$ different upper halves of a canonical semi-magic square by the following algorithm:
\begin{enumerate}
\item Write down all elements of $s_1$ into the first row of a semi-magic square in increasing order.
\item Write down all elements of $s_2$ into the second row in one of $6! = 720$ ways.
\item Write down all elements of $s_3$ into the third row in one of $6! = 720$ ways.
\item Calculate sums over all upper triplets.
\item Sort columns by the sums of triplets in increasing order. In case of a tie sort them by the upper element.
\end{enumerate}

Let $P(X) = ( \sum_{i=1}^3 X(i,1), \sum_{i=1}^3 X(i,2), \cdots, \sum_{i=1}^3 X(i,6) )$ be {\it a profile} of the upper half of a canonical square $X$.

Similarly, for each $(s_4, s_5, s_6) \in L$ we can generate $720 \times 720$ different lower halves of a canonical semi-magic square, but on step 5 we sort the columns in decreasing order. Let $P'(Y) = ( 111 - \sum_{i=4}^6 Y(i,1), 111 - \sum_{i=4}^6 Y(i,2), \cdots, 111 - \sum_{i=4}^6 Y(i,6) )$ be a profile of a lower half of a canonical square $Y$.

Any two upper and lower halves can be combined together into a valid canonical semi-magic square if they have the same profile. Let $N_m^U(p)$ and $N_m^L(p)$ be the number of upper and lower halves for set $m \in M'$ and profile $p$. Then $|C_m| = \sum_{p \in P} N_m^U(p) N_m^L(p) f(p)$, where $P$ is the set of all possible profiles, and $f(p)$ is the number of  permutations of elements in $p$ that do not change $p$.

$f(p)$ appears from the following observation: for the upper half, in case of tie for sums of triplets, we sort them in determined way; for the lower half in case of tie we may sort them in any order, and every such sorting produces a valid canonical square.

$f(p)$ can be calculated in the following way: first, we split all elements of $p$ into groups of equal numbers, and then take a product of factorials of the groups' sizes.

So, we have the following algorithm:
\begin{enumerate}
\item Generate the set $U$ of all partitions of $m$ into series $s_1$, $s_2$ and $s_3$.
\item For all $(s_1, s_2, s_3) \in U$ generate all upper halves, calculate their profiles and for every appeared profile determine the number of upper halves having such profile using a hash table.
\item Iterate over all elements of the hash table and for every key $p$ multiply its value by $f(p)$.
\item Generate the set $L$ of all partitions of $A \backslash m$ into series $s_4$, $s_5$ and $s_6$.
\item For all $(s_4, s_5, s_6) \in L$ generate all lower halves, calculate their profiles $p$ and get a sum of values from the hash table found by keys $p$.
\end{enumerate}

\subsection{Technical notes}

Any implementation of the algorithm described in the previous subsection may calculate $|C_m|$ for any class $m \in M'$ in several minutes. Our implementation may handle each of them in about $15$-$20$ seconds because of special technical optimizations. We describe some details about them here.

\subsubsection{Hash table}

A hash table is a data structure for fast searching for stored {\it values} by {\it keys}. In our case keys are the profiles and values are the numbers of square halves with corresponding profiles.

Consider a profile $p = (p_1, p_2, \cdots, p_6)$. We have $p_1 \le p_2 \le p_3 \le p_4 \le p_5 \le p_6$, and $\sum_{i=1}^6 p_i = 333$. By the pigeonhole principle, $p_1 < 56$, $p_2 < 67$ and $p_3 < 84$. Also, $p_4 < 106$ and $p_5 < 106$ because $36 + 35 + 34 = 105$. $p_6 = 333 - \sum_{i=1}^5 p_i$. So, any profile can be encoded as an integer less than $56 \times 67 \times 84 \times 106 \times 106 = 3\,541\,227\,648 < 2^{32}$. This encoding allows us to use as key an unsigned 32-bit integer. Due to a minor mistake, there is $85$ in our code instead of $84$, but it changes nothing.

For the values we can use unsigned 32-bit integers too. For every class $|C_m|$ the total number of insertions into the hash table is exactly the number of partitions of $m$ into three series multiplied by $720 \times 720$. The maximal number of partitions over all classes is $2704$, so the maximum possible value is $2704 \times 720 \times 720 = 1\,401\,753\,600 < 2^{32}$, and there is no overflow even if all square halves over all partitions have the same profile.

There may be an overflow on the 3rd step of the main algorithm, but we can consider the maximum value in the hash table before this step to verify that there will be no overflows here. The maximum value over all classes was $68\,225$, so even if it is multiplied by $6! = 720$, it still fits into 32-bit integer. Actually, the maximum value we have got after multiplication was $86\,328$.

Let $k$ be an encoded profile $p$. The best hash function of those we tried was $h(k) = k \land (2^{24}-1)$, or just the last 24 bits of $k$. It is very fast for computing, and keys $k$ we try to insert into the hash table are different enough to make insignificant number of collisions. We used a simple hash table with open addressing and linear probing. The size of the hash table is $2^{24} + 1024$, where $1024$ is a ``tail'' in the end: our hash table is not cycled from end to beginning because checking whether we need to jump to the beginning during the probing requires additional time. We check for the ``hash table overflow'' only once: before the 3rd step of the main algorithm. Cell $2^{24} + 512$ was never reached.

The last trick we used here was partial sorting of queries before inserting and searching. For every partition of $m$ we have $720 \times 720$ queries. Before performing these queries we sort them by the first $12$ bits of $h(k)$ using counting sort. Number $12$ is chosen because the modern computers typically have $2^{12}$ lines in the L2 cache, so the counting sort works very fast. The partially sorted list of queries significantly reduces the number of cache misses (again, because the L2 cache has $2^{12}$ lines, it almost always covers $2^{12}$ consecutive cells of the hash table), and it gives a speedup of about 8 times.

\subsubsection{Sorting six integers}

For every partition of $m$ (and $A \backslash m$) we need to sort six integers $720 \times 720$ times. To do it fast we use sorting networks:
$$
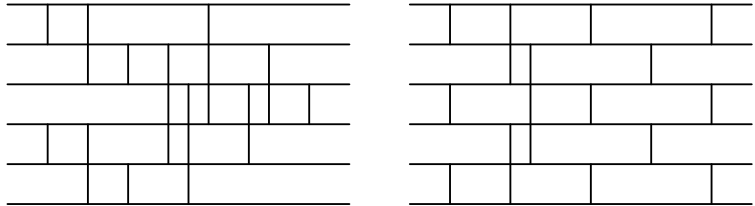
\includegraphics[scale=1.5]{sortnet.mps}$$

One may use the left sorting network that requires only $12$ comparators, and implement them directly as compares and swaps. But we use the right sorting network with $14$ comparators, because it can be processed in $5$ passes instead of $6$, if we process independent comparators in parallel.

To process comparators in parallel, we use Streaming SIMD Extensions (SSE). The numbers we sort are small, so we pack all them into $6$ bytes. Because SSE instructions may operate on $128$-bit words ($16$ bytes), we actually process two sorting networks in parallel, sorting $2$ sets of $6$ numbers. One may use AVX2 instructions (they operate on $256$-bit words) to sort up to $5$ sets of $6$ numbers in parallel manner, but we didn't implement it.

We present the code of such sorting in Listing~\ref{list_sse}, details are left to the reader.

\section{Theoretical investigations} \label{sec_theo}

First of all, we need a fact that there are no panmagic squares of order 6. It has been proved by Andrew Hollingworth Frost~\cite{frost1878} and later by Charles Planck~\cite{planck1919}, as it is mentioned in the Introduction. The latter proof is so short and simple, so we describe it here.

\begin{lemma}[Charles Planck] \label{lemma_pan}
There are no panmagic squares of order 6.
\end{lemma}
\begin{proof}
Let $P$ be a panmagic square, consider the following marking of its cells:
$$
\left[\begin{matrix}
a & b & a & b & a & b \\
c &  & c &  & c & \\
a & b & a & b & a & b \\
c &  & c &  & c & \\
a & b & a & b & a & b \\
c &  & c &  & c & \\
\end{matrix}\right]
$$

Let $A$ be the sum over cells marked with ``$a$'', $B$ be the sum over ``$b$''-cells, and $C$ be the sum over ``$c$''-cells. Then $A+B = 3Z$ for rows 1, 3, and 5, $A+C = 3Z$ for columns 1, 3, and 5, and $B+C = 3Z$ for three broken diagonals (either of ``$\backslash$''-type or ``$\slash$''-type), where $Z = 6(6^2+1)/2 = 111$ is the magic constant. Adding these three equations we get $2(A+B+C) = 9Z$. In the resulting equation the left side is even, while the right side is odd.
\end{proof}

Now we describe transforming of a semi-magic square into \textit{a normalized} form. First, we permute rows and columns in a such way, that number $1$ is placed in the upper left corner. Then we additionally permute rows and columns that numbers in the first row and numbers in the first column are sorted in ascending order. Let $X$ be a square we have got on this step; if $X(1,2) > X(2,1)$, we transpose the square. The following picture shows an example of such transforming:
$$
\left[\begin{matrix}
35 & 12 & 18 & 4 & 9 & 33 \\
34 & 15 & 29 & 1 & 26 & 6 \\
7 & 20 & 24 & 16 & 23 & 21 \\
19 & 28 & 2 & 32 & 25 & 5 \\
13 & 14 & 30 & 27 & 17 & 10 \\
3 & 22 & 8 & 31 & 11 & 36 \\
\end{matrix}\right]
\to
\left[\begin{matrix}
1 & 34 & 15 & 29 & 26 & 6 \\
4 & 35 & 12 & 18 & 9 & 33 \\
16 & 7 & 20 & 24 & 23 & 21 \\
32 & 19 & 28 & 2 & 25 & 5 \\
27 & 13 & 14 & 30 & 17 & 10 \\
31 & 3 & 22 & 8 & 11 & 36 \\
\end{matrix}\right]
\to
$$
$$
\to
\left[\begin{matrix}
1 & 6 & 15 & 26 & 29 & 34 \\
4 & 33 & 12 & 9 & 18 & 35 \\
16 & 21 & 20 & 23 & 24 & 7 \\
27 & 10 & 14 & 17 & 30 & 13 \\
31 & 36 & 22 & 11 & 8 & 3 \\
32 & 5 & 28 & 25 & 2 & 19 \\
\end{matrix}\right]
\to
\left[\begin{matrix}
1 &  4 & 16 & 27 & 31 & 32 \\
6 & 33 & 21 & 10 & 36 & 5 \\
15 & 12 & 20 & 14 & 22 & 28 \\
26 & 9 & 23 & 17 & 11 & 25 \\
29 & 18 & 24 & 30 &  8 & 2 \\
34 & 35 &  7 & 13 & 3 & 19 \\
\end{matrix}\right]
$$

Let $\operatorname{norm}(X)$ be a semi-magic square $X$ in the normalized form.

Let $N$ be the set of all normalized semi-magic squares of order 6. Then, obviously, $|C| = 2|N|$, where $C$ is the set of all canonical semi-magic squares of order 6.

We also consider \textit{a complementing} operation which replaces every number in a square with its complementing number, i.e. every number $x$ is replaced with $37-x$ for order 6 (or by $n^2+1-x$ for order $n$). Let $\operatorname{comp}(X)$ be \textit{a complement} of square $X$. If $X$ is semi-magic, then $\operatorname{comp}(X)$ is semi-magic too.

Now we combine these operations and get a complement-and-normalize operation which operates over $N$: $\operatorname{cn} := \operatorname{norm} \cdot \operatorname{comp} : N \to N$. $\operatorname{cn}$ is an involution, i.e. $\operatorname{cn}(\operatorname{cn}(X)) = X$ for all $X \in N$.
For some normalized semi-magic squares $X$ we have $\operatorname{cn}(X) = X$. We call such squares \textit{self-complement}.

\begin{lemma} \label{lemma_trans}
For any self-complement square $X$ we do not use the transpose operation while performing $\operatorname{cn}(X)$.
\end{lemma}
\begin{proof}
Suppose we used the transposing operation and some row $i$ of $X$ is transformed into some column $j$ of $X$. Consider a number $x = X(i,j)$. Then the number $37-x$ should be in the row $i$ and in the column $j$ too, but the only cell where it can be placed is the intersection of the row $i$ and the column $j$ which is already filled with different number $x$.
\end{proof}

By Lemma~\ref{lemma_trans}, after performing $\operatorname{cn}$ over self-complement squares, rows are transformed into rows, and columns are transformed into columns. Because $\operatorname{cn}$ is involution, the sets of rows and columns can be split into pairs and immovable points under $\operatorname{cn}$. We call two different rows/columns \textit{complementing} if they are transformed one into another. If a row or a column is the immovable point, we call such row/column \textit{self-complement}.

\begin{lemma} \label{lemma_sc}
Any self-complement square $X$ does not have a row and a column which are both self-complement.
\end{lemma}
\begin{proof}
Suppose a row $i$ of $X$ and a column $j$ of $X$ are both self-complement. Consider a number $x = X(i,j)$. Then, similarly to Lemma~\ref{lemma_trans}, number $37-x$ should be in the row $i$ and in the column $j$ both, which is impossible.
\end{proof}

Lemma~\ref{lemma_sc} gives us the following information: in a self-complement semi-magic square either all rows or all columns are split into pairs under $\operatorname{cn}$.

Now we can state the following

\begin{theorem} \label{theo_two}
$|N|$ is divisible by 2.
\end{theorem}
\begin{proof}
\textbf{Case 1.} All $X \in N$ which are not self-complement can be split into pairs $(X, \operatorname{cn}(X))$, so the number of them is even.

\medskip

\textbf{Case 2.} Any self-complement square $X \in N$ which has no self-complement rows and self-complement columns can be transformed into a panmagic square permuting rows and columns in order $(a,b,c,a,b,c)$, where the same letters represent pairs of complementing rows/columns. By Lemma~\ref{lemma_pan}, the number of such squares is $0$.

\medskip

\textbf{Case 3.} Now we have to consider a set $N' \subseteq N$ of self-complement squares which have either a self-complement row or a self-complement column. Instead we consider a set $N^{*}$ which can be obtained from $N'$ by transposing those squares which have a self-complement column. All squares $X \in N^{*}$ have a self-complement row and now we are not restricted by constraint $X(1,2) < X(2,1)$. We have $|N^{*}| = |N'|$.

Consider $X \in N^{*}$, $X$ has a self-complement row. Then $X$ has at least two self-complement rows, because 5 remaining rows cannot be split into complementary pairs. Let $i$ and $j$, $i < j$, be minimal indices of the self-complement rows. By Lemma~\ref{lemma_sc}, all columns are split into complementary pairs, and exactly two such pairs do not intersect with the first column. Let $(a,b)$ and $(c,d)$ be the indices of these pairs of columns. Then, pairs of numbers $(X(i,a), X(i,b))$, $(X(i,c),X(i,d))$, $(X(j,a),X(j,b))$, and $(X(j,c),X(j,d))$ are complimenting, i.e. sum of the numbers within any of these pairs is $37$.

\begin{itemize}
\item[] \textbf{Case 3a.} If $i>1$, then we can swap the following pairs of numbers: $(X(i,a),X(i,b))$ with $(X(j,a),X(j,b))$, or $(X(i,c),X(i,d))$ with $(X(j,c),X(j,d))$. We can also swap both of these pairs. So, we get 3 another different squares $X_1$, $X_2$ and $X_3$, every of them is a valid semi-normalized semi-magic square from $N^{*}$. Moreover, applying these swaps to any square from the set $\{ X, X_1, X_2, X_3 \}$ produces three remaining squares. So, the number of squares of such kind is divisible by 4.

\item[] \textbf{Case 3b.} If $i=1$, then $a=2$, $b=5$, $c=3$ and $d=4$, because the first row is self-complement and numbers in it are sorted in ascending order. We can swap the same pairs as in case 3a, but after that we have to permute columns to get the first row to be sorted. Luckily, it is always possible by swapping columns $a$ with $b$, $c$ with $d$, and $(a,b)$ with $(c,d)$. And again, we have a set of 4 squares from $N^{*}$, any of them will be transformed into another one from this set, so the number of squares of such kind is also divisible by 4.
\end{itemize}

Thus, $|N^{*}|$ is divisible by 4, which implies $|N'|$ is divisible by 4 too.
\end{proof}

\begin{cor}
$|C|$ is divisible by 4.
\end{cor}

One may try to expand the proof of Theorem~\ref{theo_two} to prove that $|N|$ is divisible by $4$. To do that we may try to split pairs $(X, \operatorname{cn}(X))$ from Case 1 into pairs of pairs by some swapping operations like ones such in Cases 3a and 3b. But it seems like simple swapping operations do not work here because there exists the following semi-magic square found by a computer search:
$$
K = \left[\begin{matrix}
 1 & 9 & 14 & 20 & 32 & 35 \\
10 & 23 & 36 &  7 & 16 & 19 \\
15 & 33 & 24 &  5 & 12 & 22 \\
26 & 11 &  3 & 34 & 31 &  6 \\
29 & 27 & 21 & 28 &  2 &  4 \\
30 &  8 & 13 & 17 & 18 & 25 \\
\end{matrix}\right]
$$

It has no pairs of numbers that can be swapped, i.e. for all $i \ne j$, and all $a \ne b$, we have $K(i,a)+K(j,a) \ne K(i,b)+K(j,b)$ and $K(i,a)+K(i,b) \ne K(j,a)+K(j,b)$. It also has no triples that can be swapped, i.e. for all $i,j,k$ ($i \ne j$, $i \ne k$, $j \ne k$), and all $a \ne b$, we have $K(i,a)+K(j,a)+K(k,a) \ne K(i,b)+K(j,b)+K(k,b)$ and $K(a,i)+K(a,j)+K(a,k) \ne K(b,i)+K(b,j)+K(b,k)$. Moreover, $K$ does not have a configuration which can be swapped in the following way:
$$
\left[\begin{matrix}
\alpha+C & & \beta \\
\alpha & \beta+C & \\
 & \gamma & \gamma+C \\
\end{matrix}\right]
\to
\left[\begin{matrix}
\alpha & & \beta+C \\
\alpha+C & \beta & \\
 & \gamma+C & \gamma \\
\end{matrix}\right]
$$

The latter swapping scheme can be defined more formally: for all pairwise distinct $i,j,k$, and all pairwise distinct $a,b,c$, the equality $K(i,a)-K(j,a)=K(j,b)-K(i,c)=K(k,c)-K(k,b)$ does not hold. Furthermore, this scheme does not work even if we rearrange $\alpha$, $\beta$ and $\gamma$ in the following ways:
$$
\left[\begin{matrix}
\alpha & & \gamma \\
\alpha & \beta & \\
 & \beta & \gamma \\
\end{matrix}\right]
\quad
\left[\begin{matrix}
\alpha & & \gamma \\
\beta & \beta & \\
 & \alpha & \gamma \\
\end{matrix}\right]
\quad
\left[\begin{matrix}
\alpha & & \beta \\
\gamma & \beta & \\
 & \alpha & \gamma \\
\end{matrix}\right]
\quad
\left[\begin{matrix}
\alpha & & \alpha \\
\beta & \beta & \\
 & \gamma & \gamma \\
\end{matrix}\right]
\quad
\left[\begin{matrix}
\alpha & & \alpha \\
\gamma & \beta & \\
 & \beta & \gamma \\
\end{matrix}\right]
$$

The squares of such kind are very rare. We were able to find only three more examples after a day of search:
$$\small
\left[\begin{matrix}
 1 &  4 &  8 & 29 & 33 & 36 \\
13 & 25 & 32 & 30 &  5 &  6 \\
16 & 21 & 31 &  7 & 19 & 17 \\
26 & 24 & 22 & 10 & 11 & 18 \\
27 & 35 &  3 & 23 &  9 & 14 \\
28 &  2 & 15 & 12 & 34 & 20 \\
\end{matrix}\right]
\quad
\left[\begin{matrix}
 1 &  7 & 16 & 22 & 30 & 35 \\
10 &  9 & 34 & 27 & 29 &  2 \\
18 & 32 &  4 & 23 &  3 & 31 \\
20 & 33 & 21 & 14 & 15 &  8 \\
26 & 25 & 17 & 13 &  6 & 24 \\
36 &  5 & 19 & 12 & 28 & 11 \\
\end{matrix}\right]
\quad
\left[\begin{matrix}
 1 &  8 & 19 & 22 & 27 & 34 \\
14 & 21 & 29 & 18 &  3 & 26 \\
16 & 20 &  6 &  5 & 33 & 31 \\
23 & 24 & 12 & 36 &  9 &  7 \\
25 & 28 & 30 & 13 &  4 & 11 \\
32 & 10 & 15 & 17 & 35 &  2 \\
\end{matrix}\right]
$$

\section{Probabilistic estimation} \label{sec_prob}

For the probabilistic estimation of $|Q|$ we used a method introduced by Donald Ervin Knuth~\cite{Knuth75} in a very straightforward way. We fill the square in the following order:
$$
\left[\begin{matrix}
a_1 & a_2 & a_3 & a_4 & a_5 & a_6 \\
b_1 & c_1 & c_2 & c_3 & c_4 & c_5 \\
b_2 & c_6 & c_{10} & c_{11} & c_{12} & c_{13} \\
b_3 & c_7 & d_1 & d_4 & d_5 & d_6 \\
b_4 & c_8 & d_2 & d_7 & d_9 & d_{10} \\
b_5 & c_9 & d_3 & d_8 & d_{11} & d_{12} \\
\end{matrix}\right]
$$

Let $S_1$ be a set of series of order $6$ that contain number $1$, $|S_1|=4\,739$.

First, we iterate over $S_1$ and fill cells $a_1,\cdots,a_6$. Then, we iterate over $S_1$ again and fill cells $b_1,\cdots,b_5$. Of course, on the second step we consider only those series which have no common elements with the already placed numbers in cells $a_1,\cdots,a_6$ except number $1$. We also consider only those possibilities where $a_1 < a_2 < \cdots < a_6$, $a_1 < b_1 < \cdots < b_5$ and $a_2 < b_1$. It gives $4\,531\,580$ ways to fill cells $a_1,\cdots,b_5$.

For each of $4\,531\,580$ ways, we randomly fill cells $c_1,\cdots,c_{13}$ by the remaining numbers in order of indices, i.e. first, we fill cell $c_1$ by one of $25$ remaining numbers (they all have equal probability to be chosen), then $c_2$ by one of $24$ remaining numbers and so on. For cells $c_5$, $c_9$ and $c_{13}$ we have at most one way to fill. This part gives us a probabilistic factor
$$F = 25 \times 24 \times 23 \times 22 \times 20 \times 19 \times 18 \times 16 \times 15 \times 14 = 6\,977\,456\,640\,000.$$

Finally, we try to fill cells $d_1,\cdots,d_{12}$ by the remaining numbers in all possible ways, iterating over several nested loops.

Let the random variable $\xi$ be the number of semi-magic squares that is computed by the described algorithm. Then
$|Q| = \frac{(6!)^2}{4} F \operatorname{E}[\xi]$.

To estimate $\operatorname{E}[\xi]$ we consider another random variable $\xi' = \frac{1}{n} \sum_{i=1}^n \xi_i$, where $\xi_1, \cdots , \xi_n$ are independent copies of $\xi$. Then, by the central limit theorem $\xi'$ approximates the normal distribution with mean $\operatorname{E}[\xi]$ and variance $\operatorname{Var}[\xi]/n$ when $n \to \infty$. Let $\xi^*_1, \cdots , \xi^*_n$ be measures of random variables $\xi_1, \cdots , \xi_n$, then $\operatorname{E}[\xi'] \approx \frac{1}{n} \sum_{i=1}^n \xi^*_i$ and $\operatorname{Var}[\xi'] \approx
\frac{1}{n}\left( \frac{1}{n} \sum_{i=1}^n (\xi^*_i)^2 - \left(\frac{1}{n} \sum_{i=1}^n \xi^*_i \right)^2 \right)$.

\section{Results}\label{sec_results}

We implemented the algorithm described in Section~\ref{sec_exact} and computed the total number of canonical semi-magic squares of order $6$:
$$|C| = 1\,459\,732\,411\,194\,444\,392 = 2^3 \times 41 \times 4\,450\,403\,692\,665\,989.$$

The result has been double-checked, i.e. for every class we have at least two runs with the same answer. It has been partially confirmed by a similar program implemented by Walter Trump. Since that program is much slower, we only used it to double-check about $1\,000$ classes of $9\,366\,138$ chosen at random. Also, the number we have got is divisible by $4$; it matches our theoretical result from Section~\ref{sec_theo}.

To handle all classes we split them into $100$ jobs. All classes $m \in M'$ were ordered by value $\sum_{i \in m}2^{36-i}$. Job $i \in \{ 0, 1, \cdots , 99 \}$ contains all classes with ids $i$ modulo $100$ from the ordered list. The total number of semi-magic squares for every job can be found in Table~\ref{table_jobs}. We also present the number of semi-magic squares for several selected classes in Table~\ref{table_classes}. One may use this data for independent verification.

The program was run on machines with the following specifications: Intel Core i7-5820K 4.2GHz with 6 cores 32Gb DDR4 RAM and Intel Core i7-3770 3.4GHz with 4 cores 12Gb DDR3 RAM. The double-check was done with help of machines of Walter Trump, which were two nearly identical computers (Intel Core i7-8700K 3.7GHz with 6 cores and 16Gb DDR4 RAM each). Running one job on the i7-5820K in one thread takes about $200$ hours, while running it in $10$ threads takes about $370$ hours because of heavy utilization of RAM bandwidth. Nevertheless, the average time per job decreases from $200$ to just $37$ hours, so this approach is acceptable. DDR4 memory is very important here, because on our other machine i7-3770 with DDR3 memory we had no significant improvement of the average time by running the program in more than $4$ threads. The total time of the computation was about $5$ months of real time (about $5$ CPU core-years), for the double-check we have spent another $3$ months of real time.

After multiplying by $(6!)^2/8$ we get the total number of semi-magic squares of order $6$ up to reflections and rotations:
$$|Q| = 94\,590\,660\,245\,399\,996\,601\,600.$$

We also implemented the probabilistic algorithm described in Section~\ref{sec_prob}. We have done $10^7$ measures (it has taken about 1 month of real time) and got the following estimation:
$$(0.9459104 \pm 0.0000072) \times 10^{23},$$
where the range corresponds to $3 \sigma$ error (the exact value should be within the range with probability $99.7\%$).
We used RdRand here --- an Intel on-chip hardware random number generator which provides a high quality pseudo-random number sequence. Our exact result fits into the probabilistic estimated range.

\section{Conclusion} \label{sec_conc}

We have calculated the number of semi-magic squares of order 6. We will be grateful if someone confirmed our result.

We have proved that the number of canonical semi-magic squares is divisible by 4, which matches our computational results. Since our computed number is divisible by 8, it would be interesting if one found an analytic proof for this constraint.

We have also given the probabilistic estimation of the desired number which matches the exact result. Our estimation differs from the estimation
$(0.9459156 \pm 0.0000051) \times 10^{23}$
by Walter Trump, and we attribute the difference to the better quality of the used random generator. Perhaps, it is a good idea to revise the other estimated values on~\cite{trump2018}.

We hope that our ideas may be applied to computing the number of magic squares of order 6, panmagic and associative magic squares of order 7.

\section*{Acknowledgments}
The author would like to thank Walter Trump for productive discussions and provided computing power.

\bibliographystyle{plain}

\newpage

\appendix
\section{Appendix}

\begin{center}
\refstepcounter{listingcounter}
Listing~\arabic{listingcounter}: Implementation of sorting six integers using SSE \label{list_sse}
\end{center}
\begin{footnotesize}
\begin{verbatim}
inline __m128i simd_sort_6( __m128i vec )
{
  __m128i passA_shu = _mm_set_epi32( 0x0F0E0C0D, 0x0A0B0809, 0x07060405, 0x02030001 );
  __m128i passA_add = _mm_set_epi32( 0x0F0E0D0D, 0x0B0B0909, 0x07060505, 0x03030101 );

  __m128i pass2_shu = _mm_set_epi32( 0x0F0E0B09, 0x0D080C0A, 0x07060301, 0x05000402 );
  __m128i pass2_add = _mm_set_epi32( 0x0F0E0D0C, 0x0D0A0C0A, 0x07060504, 0x05020402 );
  __m128i pass2_and = _mm_set_epi32( 0x0000FEFD, 0xFEFEFDFE, 0x0000FEFD, 0xFEFEFDFE );

  __m128i pass4_shu = _mm_set_epi32( 0x0F0E0D0B, 0x0C090A08, 0x07060503, 0x04010200 );
  __m128i pass4_add = _mm_set_epi32( 0x0F0E0D0C, 0x0C0A0A08, 0x07060504, 0x04020200 );

  __m128i tmp = vec;
  tmp = _mm_shuffle_epi8( tmp, passA_shu );
  tmp = _mm_cmpgt_epi8( tmp, vec );
  tmp = _mm_add_epi8( tmp, passA_add );
  vec = _mm_shuffle_epi8( vec, tmp );

  tmp = vec;
  tmp = _mm_shuffle_epi8( tmp, pass2_shu );
  tmp = _mm_cmpgt_epi8( tmp, vec );
  tmp = _mm_and_si128( tmp, pass2_and );
  tmp = _mm_add_epi8( tmp, pass2_add );
  vec = _mm_shuffle_epi8( vec, tmp );

  tmp = vec;
  tmp = _mm_shuffle_epi8( tmp, passA_shu );
  tmp = _mm_cmpgt_epi8( tmp, vec );
  tmp = _mm_add_epi8( tmp, passA_add );
  vec = _mm_shuffle_epi8( vec, tmp );

  tmp = vec;
  tmp = _mm_shuffle_epi8( tmp, pass4_shu );
  tmp = _mm_cmpgt_epi8( tmp, vec );
  tmp = _mm_add_epi8( tmp, pass4_add );
  vec = _mm_shuffle_epi8( vec, tmp );

  tmp = vec;
  tmp = _mm_shuffle_epi8( tmp, passA_shu );
  tmp = _mm_cmpgt_epi8( tmp, vec );
  tmp = _mm_add_epi8( tmp, passA_add );
  vec = _mm_shuffle_epi8( vec, tmp );

  return vec;
}
\end{verbatim}
\end{footnotesize}

\begin{table}[h]
\caption{The number of semi-magic squares for jobs 0 to 99}
\label{table_jobs}\centering\small
\begin{tabular}{ccccc}
$n$ & job $n$ & job $25+n$ & job $50+n$ & job $75+n$ \\
\hline
0 & 14595299890506839 & 14595316001632561 & 14603838454980554 & 14597617136058905 \\
1 & 14601348238542799 & 14592600020669980 & 14597154681426208 & 14594419938500704 \\
2 & 14595122249331594 & 14590850126391476 & 14601782089471209 & 14594598827789990 \\
3 & 14596940214796891 & 14598743550103225 & 14598589163885555 & 14593417451849138 \\
4 & 14599919524740283 & 14603418585477771 & 14594645029824771 & 14594675566340127 \\
5 & 14599481526331908 & 14603295648952067 & 14595698031939367 & 14591859025451015 \\
6 & 14595366939968279 & 14601139532864673 & 14593315664487712 & 14595575761538302 \\
7 & 14599362024620115 & 14600180402339806 & 14593428111610963 & 14592695360904517 \\
8 & 14594266170470618 & 14598025045682938 & 14599924315304908 & 14592169760616029 \\
9 & 14597824058621866 & 14597136739848778 & 14596189825398666 & 14598850078311321 \\
10 & 14600702140816706 & 14597243580577133 & 14592672834493575 & 14593855470845651 \\
11 & 14598132209236365 & 14599552119329093 & 14598996127788442 & 14599515062043516 \\
12 & 14592329176974714 & 14595101645250647 & 14596387461107846 & 14597230998596059 \\
13 & 14598864956059963 & 14597768567499825 & 14594417291677297 & 14597278947677300 \\
14 & 14598924898518133 & 14594951262878059 & 14596115008722879 & 14598949230358493 \\
15 & 14596661042596811 & 14600885826046430 & 14598173733076450 & 14597608036189274 \\
16 & 14599282408599837 & 14601628694546057 & 14602391956597491 & 14598881916536282 \\
17 & 14600280769914826 & 14598596316204762 & 14603887096543285 & 14597382946944130 \\
18 & 14598823867217825 & 14600320062542234 & 14606754521682396 & 14598581293799681 \\
19 & 14599601669562344 & 14595340485332975 & 14601092008081221 & 14595897627618200 \\
20 & 14596641892207354 & 14592870870280045 & 14604383826739901 & 14598830439857727 \\
21 & 14593383204809574 & 14595079495912035 & 14600568480441382 & 14600551999569527 \\
22 & 14591809933918280 & 14595544799219361 & 14597303972141776 & 14595836255721595 \\
23 & 14593786182687448 & 14590534399923597 & 14593049907192328 & 14596385077066818 \\
24 & 14593683315820057 & 14597572613115214 & 14600364589410945 & 14597087900740793 \\
\end{tabular}
\end{table}
\begin{table}[h]
\caption{The number of semi-magic squares for several classes}
\label{table_classes}\centering\small%
\begin{tabular}{rcr}
id & $m$ & \# of squares \\
\hline
$0$ & $\{ 1,2,3,4,5,6,7,8,9,28,29,30,31,32,33,34,35,36 \}$ & $1314107173695$ \\
$1$ & $\{ 1,2,3,4,5,6,7,8,10,27,29,30,31,32,33,34,35,36 \}$ & $1211955692745$ \\
$2$ & $\{ 1,2,3,4,5,6,7,8,11,26,29,30,31,32,33,34,35,36 \}$ & $1102582841644$ \\
$9478$ & $\{ 1,2,3,4,5,6,7,20,21,22,24,26,29,30,31,33,34,35 \}$ & $218686822793$ \\
$428467$ & $\{ 1,2,3,4,5,10,11,16,17,22,23,27,28,29,30,34,35,36 \}$ & $274746381460$ \\
$511018$ & $\{ 1,2,3,4,5,11,12,16,17,19,23,25,28,30,32,34,35,36 \}$ & $278709251527$ \\
$881323$ & $\{ 1,2,3,4,6,8,14,16,20,22,23,25,27,28,31,32,35,36 \}$ & $215381510567$ \\
$1444751$ & $\{ 1,2,3,4,7,9,12,14,15,20,23,27,29,31,32,33,35,36 \}$ & $224462276339$ \\
$1555605$ & $\{ 1,2,3,4,7,10,13,14,17,21,22,27,28,30,31,33,34,36 \}$ & $223201197661$ \\
$1863531$ & $\{ 1,2,3,4,8,9,13,15,19,20,21,25,29,30,31,32,35,36 \}$ & $258672728144$ \\
$2415714$ & $\{ 1,2,3,4,9,12,14,17,19,21,22,25,27,28,29,30,34,36 \}$ & $209576369917$ \\
$2932382$ & $\{ 1,2,3,5,6,7,11,15,19,20,23,26,29,30,32,33,35,36 \}$ & $113327661905$ \\
$4382984$ & $\{ 1,2,3,5,8,12,13,16,17,21,22,23,27,29,31,32,35,36 \}$ & $122559507765$ \\
$5475946$ & $\{ 1,2,3,6,7,10,12,14,16,20,23,28,29,30,31,32,34,35 \}$ & $114247796626$ \\
$5638250$ & $\{ 1,2,3,6,7,12,13,15,18,19,22,25,27,29,31,32,35,36 \}$ & $105541186882$ \\
$5949121$ & $\{ 1,2,3,6,8,10,13,16,20,21,23,25,26,29,30,31,33,36 \}$ & $127529946811$ \\
$7169678$ & $\{ 1,2,3,7,8,11,12,16,21,22,23,24,25,26,29,32,35,36 \}$ & $118262067011$ \\
$7392628$ & $\{ 1,2,3,7,9,10,12,14,19,20,21,27,28,30,31,32,33,34 \}$ & $114958512231$ \\
$8039539$ & $\{ 1,2,3,7,13,14,16,18,19,20,22,24,25,26,28,30,31,34 \}$ & $140013241433$ \\
$9366137$ & $\{ 1,2,3,14,15,16,18,19,20,21,22,23,24,25,26,27,28,29 \}$ & $25787950205$ \\
\end{tabular}
\end{table}

\end{document}